\documentclass[11pt]{article} 
\usepackage[english]{babel}
\usepackage[a4paper,scale=0.75,hcentering]{geometry}
\usepackage[hyperfootnotes=false]{hyperref} 
\usepackage{setspace}
\hypersetup{
	colorlinks = true,
	linkcolor = {blue},
	urlcolor = {red},
	citecolor = {blue}
}
\usepackage{amsmath, amsthm, amssymb, mathrsfs, mathtools, enumitem}
\usepackage[x11names]{xcolor}

\allowdisplaybreaks

\newcounter{results}[section] 

\theoremstyle{plain}
\newtheorem{theorem}[results]{Theorem}
\newtheorem{lemma}[results]{Lemma}
\newtheorem{proposition}[results]{Proposition}

\newtheorem*{theorem*}{Theorem}
\newtheorem*{lemma*}{Lemma}
\newtheorem*{proposition*}{Proposition}
\newtheorem*{corollary*}{Corollary}
\newtheorem*{exercise*}{Exercise}
\newtheorem*{fact*}{Fact}
\newtheorem*{problem*}{Problem}
\newtheorem*{conjecture*}{Conjecture}

\theoremstyle{remark}
\newtheorem{remark}[results]{Remark}

\newtheorem*{remark*}{Remark}
\newtheorem*{question*}{Question}

\theoremstyle{definition}

\newtheorem*{definition*}{Definition}
\newtheorem*{example*}{Example}

\numberwithin{equation}{section}

\newcommand{\R}{\ensuremath{\mathbb R}} 

\renewcommand{\S}{\ensuremath{\mathbb S}} 

\newcommand{\I}{\ensuremath{\mathcal I}}
\newcommand{\J}{\ensuremath{\mathcal J}}
\makeatletter 
\DeclarePairedDelimiter{\@tmpabs}{\lvert}{\rvert}
\newcommand{\@absstar}[1]{{\@tmpabs*{#1}}}
\newcommand{\@absnostar}[2][]{{\@tmpabs[#1]{#2}}}
\newcommand{\abs}{\@ifstar\@absstar\@absnostar}
\makeatother

\makeatletter 
\DeclarePairedDelimiter{\@tmpnorm}{\lVert}{\rVert}
\newcommand{\@normstar}[1]{{\@tmpnorm*{#1}}}
\newcommand{\@normnostar}[2][]{{\@tmpnorm[#1]{#2}}}
\newcommand{\norm}{\@ifstar\@normstar\@normnostar}
\makeatother

\let\div\undefined
\DeclareMathOperator{\div}{div}

\ifdefined\comma
	\renewcommand{\comma}{\ensuremath{\, \text{, }}}
\else
	\newcommand{\comma}{\ensuremath{\, \text{, }}}
\fi

\definecolor{removed}{RGB}{255, 0, 0}
\definecolor{added}{RGB}{0,192,0}

\usepackage[hyperfootnotes=false]{hyperref} 
\usepackage{setspace}
\hypersetup{
	colorlinks = true,
	linkcolor = {blue},
	urlcolor = {red},
	citecolor = {blue}
}

\usepackage[nameinlink,capitalise,sort]{cleveref} 
\crefname{equation}{}{} 
\crefname{enumi}{}{} 

\allowdisplaybreaks

\numberwithin{equation}{section}
\begin{document}

\title{\bf Classification of positive solutions to the H\'enon-Sobolev critical systems\thanks{Supported by National Key R\&D Program of China (2023YFA1010001) and NSFC (12171265 and 12271184). E-mail addresses: zhou-yx22@mails.tsinghua.edu.cn(Zhou),    zou-wm@mail.tsinghua.edu.cn (Zou)}}

\author{{\bf Yuxuan Zhou, Wenming Zou}\\ {\footnotesize \it  Department of Mathematical Sciences, Tsinghua University, Beijing 100084, China.} }

\date{}



\maketitle

\begin{abstract}
{\small In this paper, we investigate positive solutions to the following H\'enon-Sobolev critical system:
\begin{equation*}
\begin{cases}
 -\div(|x|^{-2a}\nabla u)=|x|^{-bp}|u|^{p-2}u+\nu\alpha|x|^{-bp}|u|^{\alpha-2}|v|^{\beta}u&\text{in }\R^n\\
 -\div(|x|^{-2a}\nabla v)=|x|^{-bp}|v|^{p-2}v+\nu\beta|x|^{-bp}|u|^{\alpha}|v|^{\beta-2}v&\text{in }\R^n\\
  u,v\in D_a^{1,2}(\R^n)
  \end{cases}
\end{equation*}
    where $n\geq 3,-\infty< a<\frac{n-2}{2},a\leq b<a+1,p=\frac{2n}{n-2+2(b-a)},\nu>0$ and $\alpha>1,\beta>1$ satisfy $\alpha+\beta=p$. Our findings are divided into two parts based on the sign of the parameter $a$.

    For $a\geq 0$, we demonstrate that any positive solution $(u,v)$ is synchronized, indicating that $u$ and $v$  are constant multiples of positive solutions to the decoupled H\'enon equation:
    \begin{equation*}
        -\div(|x|^{-2a}\nabla w)=|x|^{-bp}|w|^{p-2}w.
    \end{equation*}
     Our approach involves establishing qualitative properties of the positive solutions and employing a refined ODE approach. These qualitative properties include radial symmetry, asymptotic behaviors, modified inversion symmetry, and monotonicity. Our results extend the recent work in \cite{Esp2}, which focuses on the case $a=b$.

    For $a<0$ and $b>a$, we characterize all nonnegative ground states. Specifically, by utilizing on a sharp vector-valued Caffarelli-Kohn-Nirenberg inequality, we find that any ground state is synchronized and can thus be expressed in  terms of ground states of the aforementioned decoupled H\'enon equation. Additionally, we study the nondegeneracy of positive synchronized solutions.

    This work also explores the following $k$-coupled H\'enon-Sobolev critical system:
    \begin{equation*}
    \begin{cases}
        -\div(|x|^{-2a}\nabla u_i)=\sum\limits_{j=1}^{k}\kappa_{ij}|x|^{-bp}|u_i|^{\alpha_{ij}-2}|u_j|^{\beta_{ij}}u_i\quad\text{in }\R^n\\
        u_i\in D_a^{1,2}(\R^n)\quad \forall\ 1\leq i\leq k
    \end{cases}
\end{equation*}
    where $\kappa_{ij}>0$ and $\alpha_{ij}>1,\beta_{ij}>1$ satisfy $\alpha_{ij}+\beta_{ij}=p$. It turns out that most of our earlier arguments  can be applied to this case. One remaining problem is whether similar classification results hold for $k\geq3$. By exploiting insights from \cite{Cho}, we present a uniqueness result under prescribed initial conditions.



}

\end{abstract}

\maketitle

\section{Introduction}

In this paper, we are concerned with positive solutions to the following H\'enon-Sobolev critical system:
\begin{align}\label{sys1}
     \begin{cases}
        -\div(|x|^{-2a}\nabla u)=|x|^{-bp}|u|^{p-2}u+\nu\alpha|x|^{-bp}|u|^{\alpha-2}|v|^{\beta}u&\text{in }\R^n,\\
        -\div(|x|^{-2a}\nabla v)=|x|^{-bp}|v|^{p-2}v+\nu\beta|x|^{-bp}|u|^{\alpha}|v|^{\beta-2}v&\text{in }\R^n,\\
        u,v\in D_a^{1,2}(\R^n),
    \end{cases}
\end{align}
where $n\geq 3,-\infty <a<\frac{n-2}{2},a\leq b<a+1,p=\frac{2n}{n-2+2(b-a)},\nu>0$, and $\alpha>1,\beta> 1$ satisfy $\alpha+\beta=p$. The space $D_a^{1,2}(\R^n)$ is the completion of $C_c^{\infty}(\R^n)$ with respect to the norm
\begin{equation*}
    \norm{u}_{D_a^{1,2}(\R^n)}=\left(\int_{\R^n}|x|^{-2a}|\nabla u|^2\mathrm{d}x\right)^\frac{1}{2}.
\end{equation*}
If we let $v=0$, then \eqref{sys1} reduces to the classical critical H\'enon equation:
\begin{equation}\label{eq1}
    -\div(|x|^{-2a}\nabla u)=|x|^{-bp}|u|^{p-2}u.
\end{equation}
It is well known that \eqref{eq1} is the Euler-Lagrange equation related to the classical Caffarelli-Kohn-Nirenberg inequality:
\begin{equation}\label{ckn}
    \left(\int_{\R^n}|x|^{-bp}|u|^p\mathrm{d}x\right)^{\frac{2}{p}}\leq S(a,b,n)\int_{\R^n}|x|^{-2a}|\nabla u|^2\mathrm{d}x,
\end{equation}
where $S(a,b,n)$ denotes the sharp constant. The equation \eqref{eq1} and the inequality \eqref{ckn} have been extensively studied. When $a\geq 0$, it was discoverd by Chou and Chu in \cite{Cho} that any positive solution $u$ of \eqref{eq1} takes the form (up to a translation if $a=b=0$):
\begin{equation}\label{bubb}
    u(x)=U_{\mu}(x):=\mu^{\frac{2-n-2a}{2}}U\left(\frac{x}{\mu}\right),
\end{equation}
where
\begin{equation}
\begin{aligned}\label{bub}
    &U(x)=K(a,b,n)\left(1+|x|^{\frac{2(n-2-2a)(1+a-b)}{n-2(1+a-b)}}\right)^{-\frac{n-2(1+a-b)}{2(1+a-b)}},\\
    &K(a,b,n)=\left(\frac{n(n-2-2a)^2}{n-2(1+a-b)}\right)^{\frac{n-2(1+a-b)}{4(1+a-b)}}
\end{aligned}
\end{equation}
and $\mu>0$ is a scaling factor. The case $a<0$ is by far more complicated. Based on explicit spectral estimates, Felli and Schneider \cite{Fel} identified the region:
\begin{equation*}
    a<0,\ a<b<b_{\text{FS}}(a):=\frac{n(n-2-2a)}{2\sqrt{(n-2-2a)^2+4n-4}}-\frac{n-2-2a}{2},
\end{equation*}
in which any ground state of \eqref{eq1} is not radially symmetric. Later, Lin and Wang \cite{Lin0} observed that these ground states have exactly $\mathcal{O}(n-1)$ symmetry. It was conjectured for a long time that the Felli-Schneider curve is the threshold between the symmetry and the symmetry-breaking regions. Finally, Dolbeault, Esteban and Loss \cite{Dol} gave an affirmative answer: when $a<0$ and $b_{\text{FS}}(a)\leq b<a+1$, any positive solution $u$ is radially symmetric and takes the form given in \eqref{bubb} and \eqref{bub}. We refer to \cite{Car2,Cat,Dol,Ter} and the references therein for further background and previous works.

\vskip0.1in

Utilizing ideas from \cite{Cat}, we find that the system \eqref{sys1} is equivalent to the Hardy-Sobolev doubly critical system:
\begin{equation}\label{sys2}
    \begin{cases}
        -\div(|x|^{-2\Bar{a}}\nabla \Bar{u})+\gamma|x|^{-2(1+\Bar{a})}\Bar{u}=|x|^{-\Bar{b}p}|\Bar{u}|^{p-2}\Bar{u}+\nu\alpha|x|^{-\Bar{b}p}|\Bar{u}|^{\alpha-2}|\Bar{v}|^{\beta}\Bar{u}&\text{in }\R^n\\
        -\div(|x|^{-2\Bar{a}}\nabla \Bar{v})+\gamma|x|^{-2(1+\Bar{a})}\Bar{v}=|x|^{-\Bar{b}p}|\Bar{v}|^{p-2}\Bar{v}+\nu\beta|x|^{-\Bar{b}p}|\Bar{u}|^{\alpha}|\Bar{v}|^{\beta-2}\Bar{v}&\text{in }\R^n\\
        \Bar{u},\Bar{v}\in D_{\Bar{a}}^{1,2}(\R^n)
    \end{cases}
\end{equation}
where
\begin{equation*}
    \begin{aligned}
        \Bar{a}=a+\sqrt{\lambda^2+\gamma}-\lambda,\quad \Bar{b}=b+\sqrt{\lambda^2+\gamma}-\lambda,\quad \gamma>-\lambda^2,\quad \lambda=\frac{n-2-2\Bar{a}}{2}
    \end{aligned}
\end{equation*}
and
\begin{equation}\label{trans}
    \Bar{u}(x)=|x|^{\sqrt{\lambda^2+\gamma}-\lambda}u(x),\quad \Bar{v}(x)=|x|^{\sqrt{\lambda^2+\gamma}-\lambda}v(x).
\end{equation}

Systems \eqref{sys1} and \eqref{sys2} are intricately related to various physical phenomena. They arise in the Hartree-Fock theory for a binary mixture of Bose-Einstein condensates in two hyperfine states. They also play a significant role in the study of nonlinear optics. For instance, the solutions $u$ and $v$ are linked to the components of the beam in Kerr-like photorefractive media. Further details can be found in \cite{Akh,Esr,Fra,Kiv} and the references therein.
\vskip0.1in
Mathematically, when setting $a=b=0$ and $p=2^*$, the system \eqref{sys1} has been extensively investigated. Numerous studies have focused on the properties of ground states and bound states under various assumptions on the parameters. We refer to \cite{Abd,Abd2,Abd3,Amb,Bar,Car,Che,Cla,Col,Esp,Esp2,Far,Far2,Lin,Mit,Pen,Soa,Wei} for related works. There are also lots of results concerning the classification of positive solutions. We refer to \cite{Che1,Le,Li} and the references therein for general frameworks of equivalent integral systems.
\vskip0.1in
For the case $(a,b)\neq (0,0)$, it was recently presented by Esposito, L\'opez-Soriano and Sciunzi in \cite{Esp2} that when $0\leq a=b$ and $p=2^*$, any positive solution $(u,v)$ of the system \eqref{sys1} must take the form:
\begin{equation*}
    u(x)=c_1U_\mu(x),\quad v(x)=c_2U_\mu(x),
\end{equation*}
where $\mu>0$ and $c_1>0,c_2>0$ satisfy certain restrictions.

 \vskip0.1in

 One of the primary objectives of the current  paper is to extend their results for the system \eqref{sys1} (or the system \eqref{sys2}) to encompass a larger parameter region:
 \begin{equation}\label{zwm=1}
0\leq a\leq b< a+1;\quad  a<\frac{n-2}{2}; \quad p=\frac{2n}{n-2+2(b-a)}.
\end{equation}
We mainly focus on the system \eqref{sys1}, and our findings can be easily applied to the system \eqref{sys2} via the transformation \eqref{trans}.
\begin{theorem}\label{thm1}
    Assume $a\geq 0$. Let $(u,v)\in D_a^{1,2}(\R^n)\times D_a^{1,2}(\R^n)$ be a positive solution to the system \eqref{sys1}. Then there exist constants $\mu_0>0,c_1>0$, and $c_2>0$ such that (if $b=0$, then up to a translation)
    \begin{equation}\label{re1}
        (u,v)=(c_1U_{\mu_0},c_2U_{\mu_0}),
    \end{equation}
    where $U_{\mu_0}$ is defined in \eqref{bubb} and \eqref{bub}. Moreover, $c_1$ and $c_2$ satisfy
    \begin{equation}\label{re2}
        \begin{cases}
            c_1^{p-2}+\nu\alpha c_1^{\alpha-2}c_2^\beta=1,\\
            c_2^{p-2}+\nu\beta c_1^{\alpha}c_2^{\beta-2}=1.
        \end{cases}
    \end{equation}
\end{theorem}
\begin{remark}
    In general, the number of solutions $(c_1,c_2)$ to \eqref{re2} depends heavily on the parameters $p,\alpha,\beta,\nu$. We refer to \cite{Pen} for some further discussions in the special case of $p=2^*,\nu=1$.
\end{remark}
The derivation of Theorem \ref{thm1} is divided into two steps: First, as in \cite{Esp2}, we establish some qualitative properties for the positive solutions. Then, by transforming the system \eqref{sys1} into an ODE problem in $\R$, we exploit ideas from \cite{Wei} and \cite{Esp2} to demonstrate that any positive solution is synchronized, thereby implying \eqref{re1} and \eqref{re2}.

We will prove the following three qualitative results:
\begin{theorem}[Radial symmetry]\label{thm3}
    Assume $a\geq 0$. Let $(u,v)\in D_a^{1,2}(\R^n)\times D_a^{1,2}(\R^n)$ be a positive solution to the system \eqref{sys1}. Then the functions $u$ and $v$ are radially symmetric about the origin (up to a translation if $a=b=0$).
\end{theorem}

\begin{theorem}[Asymptotic behavior]\label{thm4}
    Assume $a\geq 0$. Let $(u,v)\in D_a^{1,2}(\R^n)\times D_a^{1,2}(\R^n)$ be a positive solution to the system \eqref{sys1}. Then there exist positive constants $u_0,v_0,u_\infty,v_\infty$ such that (if $a=b=0$, then up to a translation)
    \begin{equation}\label{re3}
        \lim_{x\rightarrow 0}u(x)=u_0,\quad \lim_{x\rightarrow 0}v(x)=v_0,
    \end{equation}
    and
    \begin{equation}\label{re4}
        \lim_{x\rightarrow \infty}|x|^{n-2-2a}u(x)=u_\infty,\quad \lim_{x\rightarrow \infty}|x|^{n-2-2a}v(x)=v_\infty.
    \end{equation}
\end{theorem}

\begin{theorem}[Modified inversion symmetry]\label{thm5}
    Assume $a\geq 0$. Let $(u,v)\in D_a^{1,2}(\R^n)\times D_a^{1,2}(\R^n)$ be a positive solution to the system \eqref{sys1}. Then, possibly after a dilation $u(x)\rightarrow \tau^{\frac{n-2-2a}{2}}u(\tau x)$ and $v(x)\rightarrow\tau^{\frac{n-2-2a}{2}}v(\tau x)$ (and if $a=b=0$, also after a translation), the functions $u$ and $v$ satisfy the modified inversion symmetry:
    \begin{equation}\label{re5}
        u\left(\frac{x}{|x|^2}\right)=|x|^{n-2-2a}u(x),\quad v\left(\frac{x}{|x|^2}\right)=|x|^{n-2-2a}v(x).
    \end{equation}
    Moreover, by setting $|x|=e^{-t}$, the function $e^{-\frac{n-2-2a}{2}}u(e^{-t})$ is even in $t\in\R$ and strictly decreasing for $t>0$.
\end{theorem}
\begin{remark}
    For the decoupled equation \eqref{eq1}, properties including radial symmetry and asymptotic behaviors were obtained in \cite{Cho}, and  it was discovered in \cite{Cat} that any positive solution is symmetric under a modified inversion.
\end{remark}

A crucial tool in our proof is a generalized moving plane method introduced by Chou and Chu in \cite{Cho}. This technique traces back to the seminal works of Alexandrov and Serrin in \cite{Ale,Ser}. Thanks to the contributions from notable works such as \cite{Ber,Che,Gid}, the method has become one of the most important tools for studying the symmetry of equations. The first adaptation of this method to systems was provided by Troy in \cite{Tro}. We also refer to \cite{Bus,ZChe,Dam,Dam2,Dan,De,Esp,Esp2,Rei,Soa} for many other interesting contributions in this area.

\vskip0.1in

For the case $a<0$ and $a<b$, we focus on the nonnegative ground states of the system \eqref{sys1}. A nontrivial solution $(u,v)$ is called a ground state if, for any other nontrivial solution $(u_0,v_0)$, it holds that $E(u,v)\leq E(u_0,v_0)$. The energy functional $E:D_a^{1,2}(\R^n)\times D_a^{1,2}(\R^n)\rightarrow \R$ is defined by
\begin{equation}\label{ene}
\begin{aligned}
    E(u,v)=&\frac{1}{2}\int_{\R^n}|x|^{-2a}\left(|\nabla u|^2+|\nabla v|^2\right)-\frac{1}{p}\int_{\R^n}|x|^{-bp}(|u|^p+|v|^p+p\nu|u|^{\alpha}|v|^\beta).
\end{aligned}
\end{equation}
Our main characterization result states that:
\begin{theorem}\label{ground}
    Assume $a<0$ and $a<b$, or $a\geq 0$. Let $(u,v)\in D_a^{1,2}(\R^n)\times D_a^{1,2}(\R^n)$ be a nonnegative solution to the system \eqref{sys1}. Then $(u,v)$ is a ground state if and only if $(u,v)=(sc_1W,sc_2W)$, where $W$ is a nonnegative (in fact, positive) ground state of the decoupled H\'enon equation \eqref{eq1}. The pair $(c_1,c_2)$ is a minimum of the following function:
    \begin{equation}\label{grou}
        f(x,y)=\frac{x^2+y^2}{(x^p+y^p+p\nu x^\alpha y^\beta)^\frac{2}{p}},\quad x\geq 0,y\geq 0,x+y=1.
    \end{equation}
    The factor $s$ is a positive normalization constant such that
   \begin{equation*}
        \begin{cases}
            (sc_1)^{p-1}+\nu\alpha (sc_1)^{\alpha-1}(sc_2)^\beta=sc_1,\\
            (sc_2)^{p-1}+\nu\beta (sc_1)^{\alpha}(sc_2)^{\beta-1}=sc_2.
        \end{cases}
    \end{equation*}
    Moreover, we have
    \begin{equation*}
        E(u,v)=\left(\frac{1}{2}-\frac{1}{p}\right)f(c_1,c_2)^{\frac{p}{p-2}}S(a,b,n)^{\frac{p}{p-2}}.
    \end{equation*}
    Recall that $S(a,b,n)$ is the sharp constant in the inequality \eqref{ckn}.
\end{theorem}
\begin{remark}\label{rem}
    Generally, one cannot guarantee the positivity of $(c_1,c_2)$. In the special case where $p\nu=1$, it has been  demonstrated in \cite{Pen} that, under appropriate constraints on $\alpha$ and $\beta$, all nonnegative ground states are given by semi-trivial pairs $(W,0)$ and $(0,W)$. In our current setting, relying on basic inequalities, we are able to analyze the following three cases:
    \begin{equation}\label{cas}
        (i)\ \min\{\alpha,\beta\}<2\quad (ii)\ \min\{\alpha,\beta\}\geq2,\nu>\frac{2^{\frac{p}{2}}-2}{p}\quad(iii)\ \min\{\alpha,\beta\}\geq2,\nu\leq \frac{p-2}{2p}.
    \end{equation}
    In the first two cases, all nonnegative ground states are positive. In the case $(iii)$, all nonnegative ground states are semi-trivial. Further details will be provided after the proof of Theorem \ref{ground} in Section \ref{sec5}.
\end{remark}
One crucial component in establishing this characterization result is the following sharp vector-valued Caffarelli-Kohn-Nirenberg inequality:
\begin{theorem}\label{vckn}
    Assume $a<0$ and $a<b$, or $a\geq0$. Then for any $(u,v)\in D_a^{1,2}(\R^n)\times D_a^{1,2}(\R^n)$, we have
    \begin{equation}\label{ckn1}
        \Bar{S}(a,b,n)\left(\int_{\R^n}|x|^{-bp}(|u|^p+|v|^p+p\nu|u|^{\alpha}|v|^\beta)\right)^{\frac{2}{p}}\leq \int_{\R^n}|x|^{-2a}\left(|\nabla u|^2+|\nabla v|^2\right).
    \end{equation}
    The sharp constant $\Bar{S}$ is given by
    \begin{equation*}
        \Bar{S}(a,b,n)=S(a,b,n)\min f(x,y),
    \end{equation*}
    where the minimum is taken over $x\geq 0,y\geq 0,(x,y)\neq(0,0)$ (recall that $f$ is defined in \eqref{grou}). The equality holds precisely when $(u,v)=(c_1W,c_2W)$, where $W$ is a minimizer of the inequality \eqref{ckn} and $(c_1,c_2)$ attains the minimum of $f$.
\end{theorem}
Another issue we are concerned with is the nondegeneracy of positive synchronized solutions of the system \eqref{sys1}.
\begin{theorem}\label{non}
    Assume $a<0$ and $b_{\mathrm{FS}}(a)<b$, or $a\geq0$. Let $(u,v)$ be a positive synchronized solution to the system \eqref{sys1}. Suppose $(u,v)=(c_1W,c_2W)$, where $W$ is a positive solution to the equation \eqref{eq1}. Then $(u,v)$ is nondegenerate if and only if
    \begin{equation}\label{aaaa}
        \nu\alpha\beta c_1^{\alpha-2}c_2^{\beta}+\nu\alpha\beta c_1^{\alpha}c_2^{\beta-2}\neq p-2.
    \end{equation}
    In particular, when $\nu\leq\frac{p-2}{2\alpha\beta}$, the condition in \eqref{aaaa} always hold.
\end{theorem}
\begin{remark}
    The special case where $a=b=0$ and $\nu=\frac{1}{2^*}$ has already been investigated in \cite{Pen}.
\end{remark}
\begin{remark}
     In fact, Theorem \ref{ground} and Theorem \ref{vckn} remain vaild when $n=2,a<0,a<b<a+1$, or when $n=1,a<-\frac{1}{2},a+\frac{1}{2}<b<a+1$. Additionally, Theorem \ref{non} remains valid when $n=2,a<0,b_{\mathrm{FS}}(a)<b<a+1$. Further clarification will be provided in our subsequent proofs.
\end{remark}

In this paper, we also consider the following $k$-coupled ($k\geq 2$) H\'enon-Sobolev critical system:
\begin{equation}\label{sys3}
    \begin{cases}
        -\div(|x|^{-2a}\nabla u_i)=\sum\limits_{j=1}^{k}\kappa_{ij}|x|^{-bp}|u_i|^{\alpha_{ij}-2}|u_j|^{\beta_{ij}}u_i\quad\text{in }\R^n\\
        u_i\in D_a^{1,2}(\R^n)\quad\forall\ 1\leq i\leq k
    \end{cases}
\end{equation}
where $n\geq 3,-\infty< a<\frac{n-2}{2},a\leq b<a+1,p=\frac{2n}{n-2+2(b-a)},\kappa_{ij}>0$, and $\alpha_{ij}>1,\beta_{ij}>1$ satisfy $\alpha_{ij}+\beta_{ij}=p$. It is not hard to see that most of our arguments for the system \eqref{sys1} can be applied to this case with minor modifications.
\begin{theorem}\label{thm7}
    Assume $a\geq 0$ and $(a,b)\neq(0,0)$. Let $(u_1,\ldots,u_k)\in \left(D_a^{1,2}(\R^n)\right)^k$ be a positive solution to the system \eqref{sys3}. Then, for any $1\leq i\leq k$:

    $(1)$ $u_i$ is radially symmetric about the origin.

    $(2)$ The limits $\lim\limits_{x\rightarrow 0}u_i(x)$ and $\lim\limits_{x\rightarrow \infty}|x|^{n-2-2a}u_i(x)$ exist and are positive.

    $(3)$ There exists a constant $\tau>0$ independent of $i$ such that after the dilation $u_i(x)\rightarrow \tau^{\frac{n-2-2a}{2}}u_i(\tau x)$, the function $e^{-\frac{n-2-2a}{2}}u_i(e^{-t})$ is even in $t\in\R$ and strictly decreasing when $t>0$.

    Moreover, when $k=2$ and $\alpha_{12}=\beta_{21},\beta_{12}=\alpha_{21},\frac{\kappa_{12}}{\kappa_{21}}=\frac{\alpha_{12}}{\beta_{12}}$, there exist positive constants $\mu_0>0,c_1>0,c_2>0$ such that
    \begin{equation}\label{re6}
        (u_1,u_2)=(c_1U_{\mu_0},c_2U_{\mu_0})
    \end{equation}
    and
    \begin{equation}\label{re7}
        \begin{cases}
            \kappa_{11}c_1^{p-2}+\kappa_{12}c_1^{\alpha_{12}-2}c_2^{\beta_{12}}=1,\\
            \kappa_{22}c_2^{p-2}+\kappa_{21}c_2^{\alpha_{21}-2}c_1^{\beta_{21}}=1.
        \end{cases}
    \end{equation}
    
When $a=b=0$, the above results still hold up to a suitable translation.
\end{theorem}
\begin{theorem}
    Assume $a<0$ and $a<b$, or $a\geq 0$. Let $(u_1,\ldots,u_k)\in \left(D_a^{1,2}(\R^n)\right)^k$ be a nonnegative solution to the system \eqref{sys3}. Suppose $\alpha_{ij}=\beta_{ji}$ and $\frac{\kappa_{ij}}{\kappa_{ji}}=\frac{\alpha_{ij}}{\beta_{ij}}$ for any $1\leq i,j\leq k$. Then $(u_1,\ldots,u_k)$ is a ground state if and only if $(u_1,\ldots,u_k)=(sc_1W,\ldots,sc_kW)$, where $W$ is a nonnegative ground state to the H\'enon equation \eqref{eq1}, and $(c_1,\ldots,c_k)$ is a minimum of the following function:
    \begin{equation*}
        f(x_1,\ldots,x_k)=\frac{\sum\limits_{i=1}^k x_i^2}{\left(\sum\limits_{i,j=1}^k\kappa_{ij}x_i^{\alpha_{ij}}x_j^{\beta_{ij}}\right)^\frac{2}{p}},\quad x_1,\ldots,x_k\geq 0,\;\; \sum\limits_{i=1}^k x_i=1.
    \end{equation*}
    The number $s$ is a positive constant such that
    \begin{equation*}
        \sum\limits_{j=1}^k \kappa_{ij}(sc_i)^{\alpha_{ij}-1}(sc_j)^{\beta_{ij}}=sc_i\quad\text{for any }1\leq i\leq k.
    \end{equation*}
    Moreover, the corresponding least energy is given by
    \begin{equation*}
        \left(\frac{1}{2}-\frac{1}{p}\right)f(c_1,\ldots,c_k)^{\frac{p}{p-2}}S(a,b,n)^\frac{p}{p-2}.
    \end{equation*}
\end{theorem}

A remaining question is whether every positive solution of the system \eqref{sys3} (for $k\geq3$) is synchronized. Theorem \ref{thm7} provides various qualitative properties for positive solutions. However, it appears that the ODE techniques from \cite{Esp2,Wei} may not be applicable in this case. In this context, we present a uniqueness result under prescribed initial conditions.
\begin{theorem}[Uniqueness]\label{thm8}
    Assume $a\geq 0$.  Let $(u_1,\ldots,u_k)$ and $(v_1,\ldots,v_k)\in \left(D_a^{1,2}(\R^n)\right)^k$ be two positive solutions to the system \eqref{sys3}. Suppose $u_i$ and $v_i$ are radially symmetric about the origin for $1\leq i\leq k$. If there exists a positive constant $\theta$ such that $u_i(0)=\theta v_i(0)$ for any $1\leq i\leq k$, then $u_i\equiv \theta v_i$ for any $1\leq i\leq k$.
\end{theorem}

The organization of this paper is outlined as follows. In Section \ref{sec2}, we focus on Theorem \ref{thm3}. We introduce a generalized moving plane method along with some regularity results. In Section \ref{sec3}, based on the property of radial symmetry, we transform the system \eqref{sys1} into suitable ODE systems. Dealing with the asymptotic behaviors and modified inversion symmetry becomes easier in this setting (Theorems \ref{thm4} and \ref{thm5}). Section \ref{sec4} is devoted to establishing classification results (Theorems \ref{thm1} and \ref{thm8}). These results are built upon refined ODE estimates. Finally, Section \ref{sec5} is dedicated to the ground states (Theorem \ref{ground} and Theorem \ref{non}). Our approach involves proving a sharp vector-valued Caffarelli-Kohn-Nirenberg inequality (Theorem \ref{vckn}) and making spectrum estimates.

\section{Proof of Theorem \ref{thm3}}\label{sec2}

In this section, we study the radial symmetry property for any positive solution $(u,v)$ to the system \eqref{sys1}. We always assume $a\geq 0$ and $(a,b)\neq (0,0)$ (the case $a=b=0$ can be treated in a similar manner). Let us fix some notations needed for the moving plane method. For any $\lambda\leq 0$, we set $\Sigma_\lambda=\{x\in\R^n\ |\ x_1<\lambda\}$ and $T_\lambda=\{x\in\R^n\ |\ x_1=\lambda\}$. For any $x=(x_1,\ldots,x_n)\in\R^n$, we denote its reflection about $T_\lambda$ by $x_\lambda=(2\lambda-x_1,x_2,\ldots,x_n)$. Following ideas from \cite{Cho}, we define
\begin{equation*}
    u_\lambda(x)=\frac{|x|^a}{|x_\lambda|^a}u(x_\lambda),\quad   v_\lambda(x)=\frac{|x|^a}{|x_\lambda|^a}v(x_\lambda),
\end{equation*}
where $x\in\Sigma_\lambda\backslash\{0_\lambda\}$. It is not hard to compute
\begin{equation}\label{mov1}
    \begin{aligned}
        -\div(|x|^{-2a}\nabla u_\lambda(x))=&-\frac{|x_\lambda|^a}{|x|^a}\div(|x_\lambda|^{-2a}\nabla u(x_\lambda))\\
        &-a(n-2-2a)u(x_\lambda)\frac{|x_\lambda|^{2a+2}-|x|^{2a+2}}{|x|^{3a+2}|x_\lambda|^{3a+2}}\\
        \geq& -\frac{|x_\lambda|^a}{|x|^a}\div(|x_\lambda|^{-2a}\nabla u(x_\lambda))\\
        =&\ \frac{|x_\lambda|^a}{|x|^a}\left(|x_\lambda|^{-bp}u(x_\lambda)^{p-1}+\nu\alpha|x_\lambda|^{-bp}u(x_\lambda)^{\alpha-1}v(x_\lambda)^{\beta}\right)\\
        =&\ \frac{|x|^{(b-a)p}}{|x_\lambda|^{(b-a)p}}\left(|x|^{-bp}u_\lambda(x)^{p-1}+\nu\alpha|x|^{-bp}u_\lambda(x)^{\alpha-1}v_\lambda(x)^{\beta}\right)\\
        \geq&\  |x|^{-bp}u_\lambda(x)^{p-1}+\nu\alpha|x|^{-bp}u_\lambda(x)^{\alpha-1}v_\lambda(x)^{\beta}.
    \end{aligned}
\end{equation}
Similarly, we have
\begin{equation}\label{mov2}
    -\div(|x|^{-2a}\nabla v_\lambda(x))\geq |x|^{-bp}v_\lambda(x)^{p-1}+\nu\beta|x|^{-bp}u_\lambda(x)^{\alpha}v_\lambda(x)^{\beta-1}.
\end{equation}

Next, we give two crucial regularity results.
\begin{proposition}\label{pro1}
    Let $(u,v)\in D_a^{1,2}(\R^n)\times D_a^{1,2}(\R^n)$ be a positive solution to the system \eqref{sys1}. Then we have $(u,v)\in L^{\infty}(\R^n)\times L^{\infty}(\R^n)$.
\end{proposition}

\begin{proposition}\label{pro2}
    Suppose $u$ is a positive $C^2$ function in $\Bar{B}_1(0)\backslash \{0\}$ satisfying
    \begin{equation*}
        -\div(|x|^{-2a}\nabla u(x))\geq 0\quad \text{in }\Bar{B}_1(0)\backslash \{0\},
    \end{equation*}
    then there exists a positive constant $K$ such that
    \begin{equation*}
        u(x)\geq K\quad \text{in }\Bar{B}_1(0)\backslash \{0\}.
    \end{equation*}
\end{proposition}
The proof of Proposition \ref{pro1} relies on a standard Moser's iteration scheme, with detailed arguments available in \cite[Proposition 3.1]{Esp2}. A more robust version of Proposition \ref{pro2} was provided in \cite[Lemma 4.2]{Cho}. The proofs for both propositions are omitted here.

Note that by defining  $\hat{u}$ and $\hat{v}$ as the modified Kelvin transforms of $u$ and $v$ respectively, according to
\begin{equation}\label{kelvin}
    \hat{u}(x)=|x|^{2+2a-n}u\left(\frac{x}{|x|^2}\right),\quad \hat{v}(x)=|x|^{2+2a-n}v\left(\frac{x}{|x|^2}\right),
\end{equation}
the pair $(\hat{u},\hat{v})$ satisfies the system \eqref{sys1} in $\R^n$:
\begin{equation}\label{sys4}
    \begin{cases}
        -\div(|x|^{-2a}\nabla \hat{u})=|x|^{-bp}\hat{u}^{p-1}+\nu\alpha|x|^{-bp}\hat{u}^{\alpha-1}\hat{v}^{\beta}&\text{in }\R^n\\
        -\div(|x|^{-2a}\nabla \hat{v})=|x|^{-bp}\hat{v}^{p-1}+\nu\beta|x|^{-bp}\hat{u}^{\alpha}\hat{v}^{\beta-1}&\text{in }\R^n\\
        \hat{u},\hat{v}\in D_{a}^{1,2}(\R^n),\quad \hat{u},\hat{v}>0\ \text{in }\R^n\backslash\{0\}.
    \end{cases}
\end{equation}
It is also evident that $\hat{u}_\lambda$ and $\hat{v}_\lambda$ satisfy similar inequalities as in \eqref{mov1} and \eqref{mov2} respectively:
\begin{equation}\label{mov3}
    \begin{aligned}
        -\div(|x|^{-2a}\nabla \hat{u}_\lambda)\geq&|x|^{-bp}\hat{u}^{p-1}+\nu\alpha|x|^{-bp}\hat{u}_\lambda^{\alpha-1}\hat{v}_\lambda^{\beta}\quad\text{in }\Sigma_\lambda\backslash\{0_\lambda\},\\
        -\div(|x|^{-2a}\nabla \hat{v}_\lambda)\geq&|x|^{-bp}\hat{v}_\lambda^{p-1}+\nu\beta|x|^{-bp}\hat{u}_\lambda^{\alpha}\hat{v}_\lambda^{\beta-1}\quad\text{in }\Sigma_\lambda\backslash\{0_\lambda\}.
    \end{aligned}
\end{equation}
Moreover, from Propositions \ref{pro1} and \ref{pro2}, there exist positive constants $c_u,C_u,c_v,C_v,R_0$ such that
\begin{equation}\label{aaa}
    \frac{c_u}{|x|^{n-2-2a}}\leq\hat{u}(x)\leq \frac{C_u}{|x|^{n-2-2a}},\quad \frac{c_v}{|x|^{n-2-2a}}\leq\hat{v}(x)\leq \frac{C_v}{|x|^{n-2-2a}},
\end{equation}
whenever $|x|\geq R_0$.

Now we are ready to present the proof of Theorem \ref{thm3}. Our arguments are motivated by \cite[Theorem 1.5]{Esp2}. It is worth mentioning that in \cite{Esp2}, in order to derive radial symmetry, the authors applied the moving plane method to a suitable translated problem. In contrast, we utilize the generalized moving plane method from \cite{Cho} to address the original problem directly. Our method appears to be somewhat simpler.

\begin{proof}[Proof of Theorem \ref{thm3}]
     First, note that it suffices to show that $\hat{u}$ and $\hat{v}$ are radially symmetric about the origin. Our approach relies on several integral estimates. In the following computations, we always use $C$ to denote a constant that depends only on $n,a,b,\nu,K,c_u,C_u,c_v,C_v,R_0$. The constant $C$ may vary from line to line. We define
     \begin{equation*}
         \xi_\lambda(x):=\hat{u}(x)-\hat{u}_\lambda(x),\quad \zeta_\lambda(x):=\hat{v}(x)-\hat{v}_\lambda(x),
     \end{equation*}
     where $x\in\Sigma_\lambda\backslash\{0_\lambda\}$ and $\lambda\leq0$. We split our proof into three steps.
     \\[8pt]
     \emph{Step 1: There exists a constant $M<0$ such that $\xi_\lambda(x)\leq 0$ and $\zeta_\lambda(x)\leq 0$ for $\lambda\leq M$ and $x\in\Sigma_\lambda\backslash\{0_\lambda\}$}.

     Assume $\lambda<-2R_0$ and fix constants $0<\epsilon<1$ (small) and $R>-2\lambda$ (large). Let us take two cut-off functions $\psi_\epsilon$ and $\eta_R$ in $C_c^\infty(\R^n;[0,1])$ such that $\psi_\epsilon=0$ in $B_\epsilon(0_\lambda)$, $\psi_\epsilon=1$ outside $B_{2\epsilon}(0_\lambda)$, and $|\nabla\psi_\epsilon|\leq C\epsilon^{-1}$. Additionally, $\eta_R=1$ in $B_R(0)$, $\eta_R=0$ outside $B_{2R}(0)$ and $|\nabla \eta_R|\leq CR^{-1}$. Testing $(\xi_\lambda^+\psi_\epsilon^2\eta_R^2,\zeta_\lambda^+\psi_\epsilon^2\eta_R^2)$ in the system \eqref{sys4} and the inequalities \eqref{mov3}, and then subtracting them, we obtain
     \begin{equation*}
         \begin{aligned}
             \int_{\Sigma_\lambda}|x|^{-2a}\nabla\xi_\lambda\cdot \nabla(\xi_\lambda^+\psi_\epsilon^2\eta_R^2)\leq& \int_{\Sigma_\lambda}|x|^{-bp}(\hat{u}^{p-1}-\hat{u}_\lambda^{p-1})\xi_\lambda^+\psi_\epsilon^2\eta_R^2\\
             &+\nu\alpha\int_{\Sigma_\lambda}|x|^{-bp}(\hat{u}^{\alpha-1}\hat{v}^\beta-\hat{u}_\lambda^{\alpha-1}\hat{v}_\lambda^\beta)\xi_\lambda^+\psi_\epsilon^2\eta_R^2,\\
             \int_{\Sigma_\lambda}|x|^{-2a}\nabla\zeta_\lambda\cdot \nabla(\zeta_\lambda^+\psi_\epsilon^2\eta_R^2)\leq& \int_{\Sigma_\lambda}|x|^{-bp}(\hat{v}^{p-1}-\hat{v}_\lambda^{p-1})\zeta_\lambda^+\psi_\epsilon^2\eta_R^2\\
             &+\nu\beta\int_{\Sigma_\lambda}|x|^{-bp}(\hat{u}^{\alpha}\hat{v}^{\beta-1}-\hat{u}_\lambda^{\alpha}\hat{v}_\lambda^{\beta-1})\zeta_\lambda^+\psi_\epsilon^2\eta_R^2.
         \end{aligned}
     \end{equation*}
     This leads us to the following inequalities:
     \begin{align}\label{asd}
         \int_{\Sigma_\lambda}|x|^{-2a}|\nabla\xi_\lambda^+|^2\psi_\epsilon^2\eta_R^2\leq&-2\int_{\Sigma_\lambda}|x|^{-2a}(\nabla\xi_\lambda^+\cdot\nabla\psi_\epsilon)\xi_\lambda^+\psi_\epsilon\eta_R^2\\
         &-2\int_{\Sigma_\lambda}|x|^{-2a}(\nabla\xi_\lambda^+\cdot\nabla\eta_R)\xi_\lambda^+\psi_\epsilon^2\eta_R\nonumber\\
         &+\int_{\Sigma_\lambda}|x|^{-bp}(\hat{u}^{p-1}-\hat{u}_\lambda^{p-1})\xi_\lambda^+\psi_\epsilon^2\eta_R^2\nonumber\\
         &+\nu\alpha\int_{\Sigma_\lambda}|x|^{-bp}(\hat{u}^{\alpha-1}\hat{v}^\beta-\hat{u}_\lambda^{\alpha-1}\hat{v}_\lambda^\beta)\xi_\lambda^+\psi_\epsilon^2\eta_R^2\nonumber\\
         =&:\I_1+\I_2+\I_3+\I_4.\nonumber
     \end{align}
     And similarly,
     \begin{align}\label{asd2}
       \int_{\Sigma_\lambda}|x|^{-2a}|\nabla\zeta_\lambda^+|^2\psi_\epsilon^2\eta_R^2\leq&-2\int_{\Sigma_\lambda}|x|^{-2a}(\nabla\zeta_\lambda^+\cdot\nabla\psi_\epsilon)\zeta_\lambda^+\psi_\epsilon\eta_R^2\nonumber\\
       &-2\int_{\Sigma_\lambda}|x|^{-2a}(\nabla\zeta_\lambda^+\cdot\nabla\eta_R)\zeta_\lambda^+\psi_\epsilon^2\eta_R\nonumber\\
       &+\int_{\Sigma_\lambda}|x|^{-bp}(\hat{v}^{p-1}-\hat{v}_\lambda^{p-1})\zeta_\lambda^+\psi_\epsilon^2\eta_R^2\\
       &+\nu\beta\int_{\Sigma_\lambda}|x|^{-bp}(\hat{u}^{\alpha}\hat{v}^{\beta-1}-\hat{u}_\lambda^{\alpha}\hat{v}_\lambda^{\beta-1})\zeta_\lambda^+\psi_\epsilon^2\eta_R^2\nonumber\\
       =&:\J_1+\J_2+\J_3+\J_4.\nonumber
     \end{align}
     In the following, we aim to estimate $\I_i$ and $\J_i$ for $1\leq i\leq 4$ in turn. For $\I_1$, using Young's inequality, Proposition \ref{pro1}, and the fact that $0\leq \xi_\lambda^+\leq \hat{u}$, we have:
     \begin{equation}\label{I_1}
         \begin{aligned}
             \I_1&\leq \frac{1}{4}\int_{\Sigma_\lambda}|x|^{-2a}|\nabla\xi_\lambda^+|^2\psi_\epsilon^2\eta_R^2+4\int_{\Sigma_\lambda}|x|^{-2a}|\nabla\psi_\epsilon|^2(\xi_\lambda^+)^2\eta_R^2\\
             &\leq \frac{1}{4}\int_{\Sigma_\lambda}|x|^{-2a}|\nabla\xi_\lambda^+|^2\psi_\epsilon^2\eta_R^2+C\norm{\hat{u}}_{L^\infty(\Sigma_\lambda)}^2\int_{\Sigma_\lambda}|x|^{-2a}|\nabla\psi_\epsilon|^2\\
             &\leq \frac{1}{4}\int_{\Sigma_\lambda}|x|^{-2a}|\nabla\xi_\lambda^+|^2\psi_\epsilon^2\eta_R^2+C\epsilon^{n-2-2a}\norm{\hat{u}}_{L^\infty(\Sigma_\lambda)}^2.
         \end{aligned}
     \end{equation}
     Similarly, for $\J_1$, we have:
     \begin{equation}\label{J_1}
         \J_1\leq \frac{1}{4}\int_{\Sigma_\lambda}|x|^{-2a}|\nabla\zeta_\lambda^+|^2\psi_\epsilon^2\eta_R^2+C\epsilon^{n-2-2a}\norm{\hat{v}}_{L^\infty(\Sigma_\lambda)}^2.
     \end{equation}
     Furthermore, by Young's inequality and the H\"older inequality, we can show that
     \begin{align}\label{I_2}
             \I_2&\leq \frac{1}{4}\int_{\Sigma_\lambda}|x|^{-2a}|\nabla\xi_\lambda^+|^2\psi_\epsilon^2\eta_R^2+4\int_{\Sigma_\lambda}|x|^{-2a}|\nabla\eta_R|^2(\xi_\lambda^+)^2\psi_\epsilon^2\\
             &\leq \frac{1}{4}\int_{\Sigma_\lambda}|x|^{-2a}|\nabla\xi_\lambda^+|^2\psi_\epsilon^2\eta_R^2+4\norm{|x|^{-a}\hat{u}}^2_{L^{2^*}\left(\Sigma_\lambda\cap \left(B_{2R}\backslash B_R\right)\right)}\left(\int_{\Sigma_\lambda}|\nabla\eta_R|^n\right)^\frac{2}{n}\nonumber\\
             &\leq \frac{1}{4}\int_{\Sigma_\lambda}|x|^{-2a}|\nabla\xi_\lambda^+|^2\psi_\epsilon^2\eta_R^2+C\norm{|x|^{-a}\hat{u}}^2_{L^{2^*}\left(\Sigma_\lambda\cap \left(B_{2R}\backslash B_R\right)\right)}.\nonumber
        \end{align}
     Analogously, we deduce that
     \begin{equation}
         \J_2\leq \frac{1}{4}\int_{\Sigma_\lambda}|x|^{-2a}|\nabla\zeta_\lambda^+|^2\psi_\epsilon^2\eta_R^2+C\norm{|x|^{-a}\hat{v}}^2_{L^{2^*}\left(\Sigma_\lambda\cap \left(B_{2R}\backslash B_R\right)\right)}.
     \end{equation}
     The estimates for $\I_3$ and $\J_3$ follow similarly:
     \begin{equation}\label{I_3}
         \begin{aligned}
             \I_3\leq&\; C\int_{\Sigma_\lambda}|x|^{-bp}\hat{u}^{p-2}(\xi_\lambda^+)^2\psi_\epsilon^2\eta_R^2\\
             \leq&\; C\norm{|x|^{-b}\hat{u}}_{L^{p}(\Sigma_\lambda\cap B_{2R})}^{p-2}\norm{|x|^{-b}\xi^+_\lambda\psi_\epsilon\eta_R}_{L^{p}(\Sigma_\lambda)}^{2}\\
             \leq&\; C\norm{|x|^{-b}\hat{u}}^{p-2}_{L^{p}(\Sigma_\lambda\cap B_{2R})}\int_{\Sigma_\lambda}|x|^{-2a}|\nabla (\xi^+_\lambda\psi_\epsilon\eta_R)|^2\\
             \leq& \;C\norm{|x|^{-b}\hat{u}}^{p-2}_{L^{p}(\Sigma_\lambda\cap B_{2R})}\int_{\Sigma_\lambda}|x|^{-2a}\left(|\nabla\xi_\lambda^+|^2\psi_\epsilon^2\eta_R^2+|\nabla\psi_\epsilon|^2(\xi_\lambda^+)^2\eta_R^2+|\nabla\eta_R|^2(\xi_\lambda^+)^2\psi_\epsilon^2\right)\\
             \leq&\; C\norm{|x|^{-b}\hat{u}}^{p-2}_{L^{p}(\Sigma_\lambda\cap B_{2R})}\left(\int_{\Sigma_\lambda}|x|^{-2a}|\nabla\xi_\lambda^+|^2\psi_\epsilon^2\eta_R^2+\norm{|x|^{-a}\hat{u}}^2_{L^{2^*}\left(\Sigma_\lambda\cap \left(B_{2R}\backslash B_R\right)\right)}\right)\\
             &+C\epsilon^{n-2-2a}\norm{|x|^{-b}\hat{u}}^{p-2}_{L^{p}(\Sigma_\lambda\cap B_{2R})}\norm{\hat{u}}_{L^\infty(\Sigma_\lambda)}^2.
         \end{aligned}
     \end{equation}
     \begin{align}\label{J_3}
             \J_3\leq&\; C\norm{|x|^{-b}\hat{v}}^{p-2}_{L^{p}(\Sigma_\lambda\cap B_{2R})}\left(\int_{\Sigma_\lambda}|x|^{-2a}|\nabla\zeta_\lambda^+|^2\psi_\epsilon^2\eta_R^2+\norm{|x|^{-a}\hat{v}}^2_{L^{2^*}\left(\Sigma_\lambda\cap \left(B_{2R}\backslash B_R\right)\right)}\right)\\
             &+C\epsilon^{n-2-2a}\norm{|x|^{-b}\hat{v}}^{p-2}_{L^{p}(\Sigma_\lambda\cap B_{2R})}\norm{\hat{v}}_{L^\infty(\Sigma_\lambda)}^2.\nonumber
     \end{align}
Finally, to evaluate $\I_4$ and $\J_4$, we need the following two estimates:
\begin{align*}
    \hat{u}^{\alpha-1}\hat{v}^\beta-\hat{u}_\lambda^{\alpha-1}\hat{v}_\lambda^\beta=&\;(\hat{u}^{\alpha-1}-\hat{u}_\lambda^{\alpha-1})\hat{v}^\beta+\hat{u}_\lambda^{\alpha-1}(\hat{v}^\beta-\hat{v}_\lambda^\beta)\\
        \leq&\; C|x|^{-(n-2-2a)(p-2)}(\xi_\lambda^++\zeta_\lambda^+)\\
        \leq&\; C\min\{\hat{u},\hat{v}\}^{p-2}(\xi_\lambda^++\zeta_\lambda^+),\\
        \hat{u}^{\alpha}\hat{v}^{\beta-1}-\hat{u}_\lambda^{\alpha}\hat{v}_\lambda^{\beta-1}\leq&\; C\min\{\hat{u},\hat{v}\}^{p-2}(\xi_\lambda^++\zeta_\lambda^+),
\end{align*}
which are guaranteed by $\lambda<-2R_0$, the mean value theorem, and the estimates in \eqref{aaa}. Arguing as in \eqref{I_3} and \eqref{J_3}, we can derive:
\begin{align}\label{I_4}
        \I_4\leq& \;C\int_{\Sigma_\lambda}|x|^{-bp}\min\{\hat{u},\hat{v}\}^{p-2}(\xi_\lambda^+)^2\psi_\epsilon^2\eta_R^2+C\int_{\Sigma_\lambda}|x|^{-bp}\min\{\hat{u},\hat{v}\}^{p-2}\xi_\lambda^+\zeta_\lambda^+\psi_\epsilon^2\eta_R^2\nonumber\\
        \leq&\; C\int_{\Sigma_\lambda}|x|^{-bp}\hat{u}^{p-2}(\xi_\lambda^+)^2\psi_\epsilon^2\eta_R^2+C\int_{\Sigma_\lambda}|x|^{-bp}\hat{v}^{p-2}(\zeta_\lambda^+)^2\psi_\epsilon^2\eta_R^2\nonumber\\
        \leq&\; C\norm{|x|^{-b}\hat{u}}^{p-2}_{L^{p}(\Sigma_\lambda\cap B_{2R})}\left(\int_{\Sigma_\lambda}|x|^{-2a}|\nabla\xi_\lambda^+|^2\psi_\epsilon^2\eta_R^2+\norm{|x|^{-a}\hat{u}}^2_{L^{2^*}\left(\Sigma_\lambda\cap \left(B_{2R}\backslash B_R\right)\right)}\right)\\
        &+ C\norm{|x|^{-b}\hat{v}}^{p-2}_{L^{p}(\Sigma_\lambda\cap B_{2R})}\left(\int_{\Sigma_\lambda}|x|^{-2a}|\nabla\zeta_\lambda^+|^2\psi_\epsilon^2\eta_R^2+\norm{|x|^{-a}\hat{v}}^2_{L^{2^*}\left(\Sigma_\lambda\cap \left(B_{2R}\backslash B_R\right)\right)}\right)\nonumber\\
        &+C\epsilon^{n-2-2a}\left(\norm{|x|^{-b}\hat{u}}^{p-2}_{L^{p}(\Sigma_\lambda\cap B_{2R})}\norm{\hat{u}}_{L^\infty(\Sigma_\lambda)}^2+\norm{|x|^{-b}\hat{v}}^{p-2}_{L^{p}(\Sigma_\lambda\cap B_{2R})}\norm{\hat{v}}_{L^\infty(\Sigma_\lambda)}^2\right).\nonumber
\end{align}
Similarly, we can derive:
\begin{align}\label{J_4}
        \J_4\leq&\; C\norm{|x|^{-b}\hat{u}}^{p-2}_{L^{p}(\Sigma_\lambda\cap B_{2R})}\left(\int_{\Sigma_\lambda}|x|^{-2a}|\nabla\xi_\lambda^+|^2\psi_\epsilon^2\eta_R^2+\norm{|x|^{-a}\hat{u}}^2_{L^{2^*}\left(\Sigma_\lambda\cap \left(B_{2R}\backslash B_R\right)\right)}\right)\\
        &+ C\norm{|x|^{-b}\hat{v}}^{p-2}_{L^{p}(\Sigma_\lambda\cap B_{2R})}\left(\int_{\Sigma_\lambda}|x|^{-2a}|\nabla\zeta_\lambda^+|^2\psi_\epsilon^2\eta_R^2+\norm{|x|^{-a}\hat{v}}^2_{L^{2^*}\left(\Sigma_\lambda\cap \left(B_{2R}\backslash B_R\right)\right)}\right)\nonumber\\
        &+C\epsilon^{n-2-2a}\left(\norm{|x|^{-b}\hat{u}}^{p-2}_{L^{p}(\Sigma_\lambda\cap B_{2R})}\norm{\hat{u}}_{L^\infty(\Sigma_\lambda)}^2+\norm{|x|^{-b}\hat{v}}^{p-2}_{L^{p}(\Sigma_\lambda\cap B_{2R})}\norm{\hat{v}}_{L^\infty(\Sigma_\lambda)}^2\right).\nonumber
   \end{align}
Combining the estimates \eqref{I_1}, \eqref{I_2}, \eqref{I_3}, \eqref{I_4}, and letting $\epsilon\rightarrow 0$, $R\rightarrow\infty$, we find that \eqref{asd} reduces to
\begin{equation}\label{asdd}
\begin{aligned}
    \left(\frac{1}{2}- C\norm{|x|^{-b}\hat{u}}^{p-2}_{L^{p}(\Sigma_\lambda)}\right)\int_{\Sigma_\lambda}|x|^{-2a}|\nabla\xi_\lambda^+|^2\leq  C\norm{|x|^{-b}\hat{v}}^{p-2}_{L^{p}(\Sigma_\lambda)}\int_{\Sigma_\lambda}|x|^{-2a}|\nabla\zeta_\lambda^+|^2.
\end{aligned}
\end{equation}
Similarly, \eqref{asd2} reduces to
\begin{equation}\label{asdd2}
   \left(\frac{1}{2}- C\norm{|x|^{-b}\hat{v}}^{p-2}_{L^{p}(\Sigma_\lambda)}\right)\int_{\Sigma_\lambda}|x|^{-2a}|\nabla\zeta_\lambda^+|^2\leq  C\norm{|x|^{-b}\hat{u}}^{p-2}_{L^{p}(\Sigma_\lambda)}\int_{\Sigma_\lambda}|x|^{-2a}|\nabla\xi_\lambda^+|^2.
\end{equation}
Note that, as $\lambda$ tends to $-\infty$, $\norm{|x|^{-b}\hat{u}}^{p-2}_{L^{p}(\Sigma_\lambda)}$ and $\norm{|x|^{-b}\hat{v}}^{p-2}_{L^{p}(\Sigma_\lambda)}$ approach zero. Hence, whenever $\lambda$ is sufficiently negative, we have:
\begin{equation}\label{asdd3}
    C\norm{|x|^{-b}\hat{u}}^{p-2}_{L^{p}(\Sigma_\lambda)}\leq\frac{1}{8},\quad C\norm{|x|^{-b}\hat{v}}^{p-2}_{L^{p}(\Sigma_\lambda)}\leq\frac{1}{8}.
\end{equation}
The combination of \eqref{asdd}, \eqref{asdd2}, and \eqref{asdd3} indicates that
\begin{equation*}
    \int_{\Sigma_\lambda}|x|^{-2a}|\nabla\xi_\lambda^+|^2+\int_{\Sigma_\lambda}|x|^{-2a}|\nabla\zeta_\lambda^+|^2\leq 0.
\end{equation*}
Therefore, $\xi_\lambda^+\equiv0$ and $\zeta_\lambda^+\equiv0$, i.e., $\xi_\lambda\leq 0$ and $\zeta_\lambda\leq 0$.
\\[8pt]
\emph{Step 2: $\xi_0(x)\leq0$ and $ \zeta_0(x)\leq0$ for any $x\in\Sigma_0$.}

Set
\begin{equation*}
    \lambda_0=\sup\{a<0\ |\ \xi_\lambda(x)\leq0
    \text{ and }\zeta_\lambda(x)\leq 0\text{ for any }\lambda\leq a\text{ and }x\in\Sigma_\lambda\backslash\{0_\lambda\}\}.
\end{equation*}
We aim to demonstrate that $\lambda_0=0$. If $\lambda_0<0$, by the maximum principle, we have $\xi_{\lambda_0}<0$ and $\zeta_{\lambda_0}<0$ in $\Sigma_{\lambda_0}\backslash\{0_{\lambda_0}\}$. To derive a contradiction, we first prove that for any $0<\delta\ll 1<R_1<\infty$, there exists $\epsilon_0>0$ (possibly dependent on $\delta$ and $R_1$) such that
\begin{equation}\label{dom}
    \{\xi_\lambda>0\}\cup\{\zeta_\lambda>0\}\subset \Omega_{\delta,R_1}:=\left(\Sigma_{\lambda_0}\backslash \Bar{B}_{R_1}\right)\cup B_\delta(0_{\lambda_0})
\end{equation}
for any $\lambda_0\leq\lambda\leq\lambda_0+\epsilon_0$. Assume the contrary. Without loss of generality, we can assume the existence of a sequence of numbers $\{\tau_m\}_m$ converging to $\lambda_0$ and a sequence of points $P_m\in\Sigma_{\tau_m}\backslash\Omega_{\delta,R}$ such that $\xi_{\tau_m}(P_m)>0$. Up to a subsequence, we also assume $P_m\rightarrow P\in \Bar{\Sigma}_{\lambda_0}\backslash\Omega_{\delta,R}$. By continuity, $\xi_{\lambda_0}(P)\geq 0$, indicating that $P$ must lie on the hyperplane $T_{\lambda_0}$. The Hopf boundary lemma then implies $\frac{\partial \xi_{\lambda_0}}{\partial x_1}(P)<0$. By continuity once more, for any $(\lambda,P')$ near $(\lambda_0,P)$, it holds that $\frac{\partial \xi_{\lambda}}{\partial x_1}(P')<0$. Now we can derive a contradiction using the facts that $\xi_{\tau_m}(P_m)>0,\xi_{\tau_m}|_{T_{\tau_m}}=0$, and the mean value theorem.

In the following, assuming that \eqref{dom} holds, it suffices to check that $\xi_{\lambda}^+\equiv0$ and $\zeta_{\lambda}^+\equiv0$ in $\Omega_{\delta,R_1}$ for certain values of $\delta,R_1$, and $\epsilon_0$, and for any $\lambda_0\leq \lambda\leq\lambda_0+\epsilon_0$. Here we can argue as in \emph{Step 1}: We test the function $(\xi_\lambda^+\psi_\epsilon^2\eta_R^2,\zeta_\lambda^+\psi_\epsilon^2\eta_R^2)$ in $\Omega_{\delta,R_1}$, applying basic inequalities, making integral estimates, and finally deducing that
\begin{equation*}\int_{\Omega_{\delta,R_1}}|x|^{-2a}|\nabla\xi_\lambda^+|^2+\int_{\Omega_{\delta,R_1}}|x|^{-2a}|\nabla\zeta_\lambda^+|^2\leq 0.
\end{equation*}
The only difference here is that we cannot allow $\lambda$ to be sufficiently negative. However, it is not hard to see that by letting $\delta$ sufficiently small and $R_1$ sufficiently large, all the arguments in \emph{Step 1} still hold. Thus, we conclude that $\xi_\lambda\leq 0$ and $\zeta_\lambda\leq 0$ for any $\lambda<\lambda_0+\epsilon_0$, leading to a contradiction!
\\[8pt]
\emph{Step 3: $\hat{u}$ and $\hat{v}$ are radially symmetric about the origin.}

\emph{Step 2} tells us that $\hat{u}(x)\leq \hat{u}(x_\lambda)$ when $x_1\leq 0$. If we perform the moving plane method in the opposite direction, we can derive the reverse inequality $\hat{u}(x)\geq \hat{u}(x_\lambda)$ when $x_1\leq 0$, which means that $\hat{u}$ is symmetric about the plane $\{x_1=0\}$. Since the choice of directions does not affect our arguments, we conclude that $\hat{u}$ must be radially symmetric about the origin. Consequently, $\hat{v}$ must also exhibit radial symmetry about the origin.
\end{proof}

\section{Proofs of Theorems \ref{thm4} and \ref{thm5}}\label{sec3}
In this section, we are devoted to investigating the asymptotic behaviors and the modified inversion symmetry for any positive solution $(u,v)$ to the system \eqref{sys1}. In both proofs, it is necessary to transform the system \eqref{sys1} into a certain equivalent ODE system.
\begin{proof}[Proof of Theorem \ref{thm4}]
    Since $u$ and $v$ are radially symmetric about the origin, we can reformulate the system \eqref{sys1} as follows
    \begin{equation}\label{ode1}
        \begin{cases}
            (r^{n-1-2a}u')'+r^{n-1-bp}(u^{p-1}+\nu\alpha u^{\alpha-1}v^{\beta})=0&\text{in }(0,+\infty)\\
            (r^{n-1-2a}v')'+r^{n-1-bp}(v^{p-1}+\nu\beta u^{\alpha}v^{\beta-1})=0&\text{in }(0,+\infty)\\
            u,v>0\quad\text{in }(0,+\infty).
        \end{cases}
    \end{equation}
    From the above system, we observe that $r^{n-1-2a}u'$ is strictly decreasing in $(0,+\infty)$, which implies that $u'$ must have only one sign near 0. Consequently,  $u(r)$ is monotonic near $0$. Utilizing Proposition \ref{pro1} and Proposition \ref{pro2}, we conclude that the limit $\lim\limits_{r\rightarrow0^+}u(r)=:u_0$ exists, and $u_0$ is positive. Applying the same arguments to $\hat{u}$ defined in \eqref{kelvin}, we find that the limit $\lim\limits_{r\rightarrow\infty}u(r)r^{n-2-2a}=\lim\limits_{r\rightarrow0}\hat{u}(r)=:u_\infty$ also exists, and $u_\infty$ is positive. Analogous results hold for $v$.
\end{proof}
Before giving the proof of Theorem \ref{thm5}, we introduce the Emden-Fowler transformation:
\begin{equation}\label{emd}
    w(r,\theta)=r^{-\frac{n-2-2a}{2}}\varphi_w(t,\theta)\quad\text{with }r=|x|,\ t=-\ln(r),\ \theta\in\S^{n-1}.
\end{equation}
The correspondence between $w$ and $\varphi_w$ establishes an isometry between $D_a^{1,2}(\R^n)$ and $H^1(\R\times\S^{n-1})$. In our scenario, since $u$ and $v$ are radially symmetric about the origin, it follows that $\varphi_u$ and $\varphi_v$ depend only on $t$. The system \eqref{sys1} can thus be transformed into the following:
\begin{equation}\label{ode2}
    \begin{cases}
        -\varphi_u''+\gamma\varphi_u=\varphi_u^{p-1}+\nu\alpha\varphi_u^{\alpha-1}\varphi_v^\beta&\text{in }\R\\
        -\varphi_v''+\gamma\varphi_v=\varphi_v^{p-1}+\nu\beta\varphi_u^{\alpha}\varphi_v^{\beta-1}&\text{in }\R\\
        \varphi_u,\varphi_v\in H^1(\R),\quad \varphi_u,\varphi_v>0\ \text{in }\R,
    \end{cases}
\end{equation}
where $n\geq 3,0\leq a<\frac{n-2}{2},a\leq b<a+1,p=\frac{2n}{n-2+2(b-a)},\gamma=\left(\frac{n-2-2a}{2}\right)^2,\nu>0$, and $\alpha>1,\beta>1$ satisfying $\alpha+\beta=p$. Since the modified inversions and dilations in $\R^n$ correspond to reflections and translations in $\R$, to prove Theorem \ref{thm5}, it suffices to address the symmetry and monotonicity properties of $\varphi_u$ and $\varphi_v$. By applying the moving plane method, we follow a similar argument as in the proof of Theorem \ref{thm3}.
\begin{proof}[Proof of Theorem \ref{thm5}]
    For any $\lambda\in\R$, let $\Sigma_\lambda=\{t<\lambda\}$. The reflection of any $t\in\R$ about $\lambda$ is denoted by $t_\lambda:=2\lambda-t$. For a function $w\in H^1(\R)$, we define $w_\lambda(t)=w(t_\lambda)$. It is evident that $(\varphi_{u,\lambda},\varphi_{v,\lambda})$ also satisfies the system \eqref{ode2}. We introduce:
    \begin{equation*}
        \xi_\lambda(t):=\varphi_u(t)-\varphi_{u,\lambda}(t),\quad \zeta_\lambda(t):=\varphi_v(t)-\varphi_{v,\lambda}(t).
    \end{equation*}
    Testing the pair $(\xi_\lambda,\zeta_\lambda)$ with $(\xi_\lambda^+,\zeta_\lambda^+)$, we deduce the following identities:
    \begin{equation}\label{qwe}
        \begin{aligned}
            \int_{\Sigma_\lambda}|(\xi_\lambda^+)'|^2&=-\gamma\int_{\Sigma_\lambda}(\xi_\lambda^+)^2+\int_{\Sigma_\lambda}(\varphi_u^{p-1}-\varphi_{u,\lambda}^{p-1})\xi_\lambda^++\nu\alpha\int_{\Sigma_\lambda}(\varphi_u^{\alpha-1}\varphi_v^\beta-\varphi_{u,\lambda}^{\alpha-1}\varphi_{v,\lambda}^\beta)\xi_\lambda^+,\\
            \int_{\Sigma_\lambda}|(\zeta_\lambda^+)'|^2&=-\gamma\int_{\Sigma_\lambda}(\zeta_\lambda^+)^2+\int_{\Sigma_\lambda}(\varphi_v^{p-1}-\varphi_{v,\lambda}^{p-1})\zeta_\lambda^++\nu\beta\int_{\Sigma_\lambda}(\varphi_u^{\alpha}\varphi_v^{\beta-1}-\varphi_{v,\lambda}^{\alpha}\varphi_{v,\lambda}^{\beta-1})\xi_\lambda^+.
        \end{aligned}
    \end{equation}
    Next, we proceed as we did for $\I_3,\J_3,\I_4,\J_4$ and continue to use $C$ to denote a constant that depends only on $n,a,b,\nu,K,c_u,C_u,c_v,C_v,R_0$:
    \begin{equation}\label{qwe1}
        \begin{aligned}
            \int_{\Sigma_\lambda}(\varphi_u^{p-1}-\varphi_{u,\lambda}^{p-1})\xi_\lambda^+\leq&\; C\int_{\Sigma_\lambda}\varphi_{u}^{p-2}(\xi_\lambda^+)^2
            \leq C\norm{\varphi_u}^{p-2}_{L^\infty(\Sigma_\lambda)}\int_{\Sigma_\lambda}(\xi_\lambda^+)^2,\\
            \int_{\Sigma_\lambda}(\varphi_v^{p-1}-\varphi_{v,\lambda}^{p-1})\zeta_\lambda^+\leq&\; C\int_{\Sigma_\lambda}\varphi_{v}^{p-2}(\zeta_\lambda^+)^2
            \leq C\norm{\varphi_v}^{p-2}_{L^\infty(\Sigma_\lambda)}\int_{\Sigma_\lambda}(\zeta_\lambda^+)^2,
        \end{aligned}
    \end{equation}
    \begin{equation}\label{qwe2}
        \begin{aligned}
            \int_{\Sigma_\lambda}(\varphi_u^{\alpha-1}\varphi_v^\beta-\varphi_{u,\lambda}^{\alpha-1}\varphi_{v,\lambda}^\beta)\xi_\lambda^+=&\int_{\Sigma_\lambda}(\varphi_u^{\alpha-1}-\varphi_{u,\lambda}^{\alpha-1})\varphi_{v}^\beta\xi_\lambda^++\int_{\Sigma_\lambda}(\varphi_v^{\beta}-\varphi_{v,\lambda}^{\beta})\varphi_{u,\lambda}^{\alpha-1}\xi_\lambda^+\\
            \leq&\; C\int_{\Sigma_\lambda}\varphi_u^{p-2}(\xi_\lambda^+)^2+C\int_{\Sigma_\lambda}\varphi_v^{p-2}(\zeta_\lambda^+)^2\\
            \leq&\; C\norm{\varphi_u}^{p-2}_{L^\infty(\Sigma_\lambda)}\int_{\Sigma_\lambda}(\xi_\lambda^+)^2+C\norm{\varphi_v}^{p-2}_{L^\infty(\Sigma_\lambda)}\int_{\Sigma_\lambda}(\zeta_\lambda^+)^2,\\
            \int_{\Sigma_\lambda}(\varphi_u^{\alpha}\varphi_v^{\beta-1}-\varphi_{v,\lambda}^{\alpha}\varphi_{v,\lambda}^{\beta-1})\xi_\lambda^+\leq&\; C\norm{\varphi_u}^{p-2}_{L^\infty(\Sigma_\lambda)}\int_{\Sigma_\lambda}(\xi_\lambda^+)^2+C\norm{\varphi_v}^{p-2}_{L^\infty(\Sigma_\lambda)}\int_{\Sigma_\lambda}(\zeta_\lambda^+)^2.
        \end{aligned}
    \end{equation}
    Collecting the estimates from \eqref{qwe}, \eqref{qwe1}, and \eqref{qwe2}, we obtain
    \begin{equation}\label{qwe3}
        \int_{\Sigma_\lambda}\left(|(\xi_\lambda^+)'|^2+|(\zeta_\lambda^+)'|^2\right)\leq \left(C\norm{\varphi_u}^{p-2}_{L^\infty(\Sigma_\lambda)}+C\norm{\varphi_v}^{p-2}_{L^\infty(\Sigma_\lambda)}-\gamma\right)\int_{\Sigma_\lambda}\left((\xi_\lambda^+)^2+(\zeta_\lambda^+)^2\right).
    \end{equation}
    From \eqref{re4} and \eqref{emd}, we infer that as $\lambda$ approaches $-\infty$, $\norm{\varphi_u}_{L^\infty(\Sigma_\lambda)}+\norm{\varphi_v}_{L^\infty(\Sigma_\lambda)}$ tends to zero. Consequently, we observe that $\xi_\lambda(t)<0$ and $\zeta_\lambda(t)<0$ whenever $\lambda$ is sufficiently negative and $t<\lambda$. Analogously, we can establish that  $\xi_\lambda(t)>0$ and $\zeta_\lambda(t)>0$ if $\lambda$ is sufficiently positive and $t<\lambda$.

    Let's define
    \begin{equation*}
        \lambda_0:=\sup\{a\in\R\ |\ \xi_\lambda(t)<0,\zeta_\lambda(t)<0\text{ for any }\lambda\leq a\text{ and }t\in\Sigma_\lambda\}.
    \end{equation*}
    It suffices to derive that $\xi_{\lambda_0}=\zeta_{\lambda_0}\equiv 0$. If not, according to the strong maximum principle, we would have $\xi_{\lambda_0}(t)<0$ and $\zeta_{\lambda_0}(t)<0$ for any $t\in\Sigma_{\lambda_0}$. Since the remaining procedures are essentially the same as those in \emph{Step 2} of the proof for Theorem \ref{thm3}, we will skip  the details here.
\end{proof}

\section{Proofs of Theorems \ref{thm1} and \ref{thm8}}\label{sec4}
The objective of this section is to characterize positive solutions to the systems \eqref{sys1} and \eqref{sys3}. We commence by providing the proof for the general uniqueness result (Theorem \ref{thm8}), employing a straightforward ODE technique from \cite{Cho}. Building upon this result, we are then able to derive Theorem \ref{thm1}, using arguments that are similar but slightly simpler than those found in \cite{Esp2,Wei}.
\begin{proof}[Proof of Theorem \ref{thm8}]
    Since the system \eqref{sys3} is invariant under dilations, we may assume $\theta=1$. From Theorem \ref{thm7}, any positive solution of the system \eqref{sys3} is radially symmetric about the origin. Thus, we can rewrite this system in radial form:
    \begin{equation}\label{ode3}
        \begin{cases}
            (r^{n-1-2a}u_i')'+r^{n-1-bp}\sum\limits_{j=1}^k\kappa_{ij}u_i^{\alpha_{ij}-1}u_j^{\beta_{ij}}=0\quad\text{in }(0,+\infty)\\
            u_i\in C^{\infty}\left((0,+\infty\right))\cap C\left([0,+\infty\right)),\quad u_i>0\ \text{in }[0,+\infty)\quad\text{for }1\leq i\leq k.
        \end{cases}
    \end{equation}
    Since $u_i$ is continuous at 0, we must have $\lim\limits_{r\rightarrow0}r^{n-1-2a}u_i'=0$, which implies that
    \begin{equation}\label{int}
        u_i(r)=-\int_0^rs^{2a+1-n}\int_0^st^{n-1-bp}\sum\limits_{j=1}^k\kappa_{ij}u_i^{\alpha_{ij}-1}(t)u_j^{\beta_{ij}}(t)\mathrm{d}t\ \mathrm{d} s+u_i(0)
    \end{equation}
    for any $1\leq i\leq k$. Analogously, we obtain
    \begin{equation}\label{int1}
        v_i(r)=-\int_0^rs^{2a+1-n}\int_0^st^{n-1-bp}\sum\limits_{j=1}^k\kappa_{ij}v_i^{\alpha_{ij}-1}(t)v_j^{\beta_{ij}}(t)\mathrm{d}t\ \mathrm{d} s+v_i(0)
    \end{equation}
    for any $1\leq i\leq k$. Set
    \begin{equation*}
        \lambda=\sup\{a\in[0,+\infty)\ |\ u_i(r)=v_i(r)\ \text{for any }r\leq a,1\leq i\leq k\}.
    \end{equation*}
    It suffices to show $\lambda=+\infty$. If not, let us take $\epsilon>0$ sufficiently small. For any $\lambda< r\leq \lambda+\epsilon$, substracting \eqref{int} by \eqref{int1} gives
    \begin{equation*}
    \begin{aligned}
        u_i(r)-v_i(r)=&\int_0^rs^{2a+1-n}\int_0^st^{n-1-bp}\sum\limits_{j=1}^k\kappa_{ij}\left(v_i^{\alpha_{ij}-1}(t)v_j^{\beta_{ij}}(t)-u_i^{\alpha_{ij}-1}(t)u_j^{\beta_{ij}}(t)\right)\mathrm{d}t\ \mathrm{d}s\\
        =&\int_\lambda^rs^{2a+1-n}\int_\lambda^st^{n-1-bp}\sum\limits_{j=1}^k\kappa_{ij}\left(v_i^{\alpha_{ij}-1}(t)v_j^{\beta_{ij}}(t)-u_i^{\alpha_{ij}-1}(t)u_j^{\beta_{ij}}(t)\right)\mathrm{d}t\ \mathrm{d}s
    \end{aligned}
    \end{equation*}
    for any $1\leq i\leq k$. Then we can estimate
    \begin{equation}\label{int2}
        \max\limits_{[\lambda,\lambda+\epsilon]}|u_i-v_i|\leq C\epsilon^{2a+2-bp}\sum\limits_{j=1}^k\max\limits_{[\lambda,\lambda+\epsilon]}|u_j-v_j|,
    \end{equation}
    where $1\leq i\leq k$ and $C$ is a constant depending on $a,b,p,k,\lambda,u_j,v_j,\kappa_{jl},\alpha_{jl},\beta_{jl},\ 1\leq j,l\leq k$. If we choose $\epsilon$ small enough such that $C\epsilon^{2a+2-bp}<\frac{1}{k}$, then by summing \eqref{int2} with respect to $1\leq i\leq k$, we immediately deduce that
    \begin{equation*}
        u_i(r)\equiv v_i(r)\quad\text{for any }r\leq \lambda+\epsilon,1\leq i\leq k.
    \end{equation*}
    However, this contradicts the choice of $\lambda$.
\end{proof}
\begin{proof}[Proof of Theorem \ref{thm1}]
    From Theorem \ref{thm8}, it suffices to demonstrate that $(u(0),v(0))$ can be expressed as $\left(c_1\mu,c_2\mu\right)$ with $\mu>0$ and $(c_1,c_2)$ solving the system \eqref{re2}. Due to homogeneity, we only need to check if, by setting $L:=\frac{u(0)}{v(0)}$, the condition $f(L)=0$ holds, where $f$ is defined by
    $$f(t)=t^{p-2}+\nu\alpha t^{\alpha-2}-1-\nu\beta t^\alpha.$$
    In the following, we concentrate on the ODE system \eqref{ode2}, which is equivalent to the system \eqref{sys1}. Thanks to Theorem \ref{thm5}, we assume that $\varphi_u$ and $\varphi_v$ are symmetric about 0 and strictly decreasing in $(0,+\infty)$. Multiplying the two equations in the system \eqref{ode2} by $\varphi_v$ and $\varphi_u$, respectively, and then subtracting the results, we deduce
    \begin{equation}\label{zxc}
        (\varphi_u'\varphi_v-\varphi_u\varphi_v')'+\varphi_u\varphi_v^{p-1}f\left(\frac{\varphi_u}{\varphi_v}\right)=0.
    \end{equation}
    From the relation \eqref{emd}, we have $\lim\limits_{t\rightarrow-\infty}\frac{\varphi_u(t)}{\varphi_v(t)}=L$. If $f(L)\neq 0$, without loss of generality, we can assume $f(L)<0$. Thus, for any $t$ sufficiently negative, we have $f\left(\frac{\varphi_u(t)}{\varphi_v(t)}\right)<0$. Integrating \eqref{zxc} over $(-\infty,0]$, we obtain
    \begin{equation*}
        \int_{(-\infty,0]}\varphi_u\varphi_v^{p-1}f\left(\frac{\varphi_u}{\varphi_v}\right)=0,
    \end{equation*}
    which implies the existence of $t_0<0$ such that $f\left(\frac{\varphi_u(t_0)}{\varphi_v(t_0)}\right)=0$ and $f\left(\frac{\varphi_u(t)}{\varphi_v(t)}\right)<0$ for any $t<t_0$. Set $L_0=\frac{\varphi_u(t_0)}{\varphi_v(t_0)}$. Integrating \eqref{zxc} over $(-\infty,t]$ for $t\leq t_0$, we get
    \begin{equation}\label{zxc1}
        \varphi_u'(t)\varphi_v(t)-\varphi_u(t)\varphi_v'(t)>0.
    \end{equation}
    Next, we multiply the two equations in the system \eqref{ode2} by $\varphi_u'$ and $L_0^2\varphi_v'$, respectively, subtracting the results, and then integrating over $(-\infty,t_0]$. This leads to the following relation:
    \begin{equation}\label{zxc2}
    \begin{aligned}
         (\varphi_u')^2(t_0)-L_0^2(\varphi_v')^2(t_0)=2\int_{(-\infty,t_0]}L_0^2p^{-1}(\varphi_v^p)'-p^{-1}(\varphi_u^p)'+\nu L_0^2\varphi_u^\alpha(\varphi_v^\beta)'-\nu(\varphi_u^\alpha)'\varphi_v^{\beta}.
    \end{aligned}
    \end{equation}
    By the definition of $L_0$, it holds that
    \begin{equation}\label{zxc3}
       \int_{(-\infty,t_0]}(\varphi_u^p)'=L_0^p\int_{(-\infty,t_0]}(\varphi_v^p)'=L_0^\beta\int_{(-\infty,t_0]}(\varphi_u^\alpha\varphi_v^\beta)'.
    \end{equation}
    Combining \eqref{zxc2} and \eqref{zxc3} gives:
    \begin{equation}\label{zxc4}
    \begin{aligned}
         (\varphi_u')^2(t_0)-L_0^2(\varphi_v')^2(t_0)=&\;2\int_{(-\infty,t_0]}\left(L_0^{2-p}-1\right)L_0^\beta p^{-1}(\varphi_u^\alpha\varphi_v^\beta)'+\nu L_0^2\varphi_u^\alpha(\varphi_v^\beta)'-\nu(\varphi_u^\alpha)'\varphi_v^{\beta}\\
         =&\;2\int_{(-\infty,t_0]}\beta\left(L_0^{2-\alpha}p^{-1}-L_0^\beta p^{-1}+\nu L_0^2\right)\varphi_u^\alpha\varphi_v^{\beta-1}\varphi_v'\\
         &+2\int_{(-\infty,t_0]}\alpha\left(L_0^{2-\alpha}p^{-1}-L_0^\beta p^{-1}-\nu\right)\varphi_u^{\alpha-1}\varphi_u'\varphi_v^\beta.
    \end{aligned}
    \end{equation}
    Recalling that $f(L_0)=0$, we observe that
    \begin{equation*}
        \beta\left(L_0^{2-\alpha}p^{-1}-L_0^\beta p^{-1}+\nu L_0^2\right)+\alpha\left(L_0^{2-\alpha}p^{-1}-L_0^\beta p^{-1}-\nu\right)=0
    \end{equation*}
    and
    \begin{equation*}
        \begin{aligned}
        \beta\left(L_0^{2-\alpha}p^{-1}-L_0^\beta p^{-1}+\nu L_0^2\right)=&\;\frac{\alpha}{p}\beta\left(L_0^{2-\alpha}p^{-1}-L_0^\beta p^{-1}+\nu L_0^2\right)\\
        &-\frac{\beta}{p}\alpha\left(L_0^{2-\alpha}p^{-1}-L_0^\beta p^{-1}-\nu\right)\\
        =&\;\frac{\alpha\beta\nu}{p}\left(L_0^2+1\right).
        \end{aligned}
    \end{equation*}
    Hence \eqref{zxc4} reduces to
    \begin{equation}\label{zxc5}
        (\varphi_u')^2(t_0)-L_0^2(\varphi_v')^2(t_0)=2\frac{\alpha\beta\nu}{p}\left(L_0^2+1\right)\int_{(-\infty,t_0]}\varphi_u^{\alpha-1}\varphi_v^{\beta-1}(\varphi_u\varphi_v'-\varphi_u'\varphi_v).
    \end{equation}
    From the estimates \eqref{zxc1} and  \eqref{zxc5}, we get $(\varphi_u')^2(t_0)-L_0^2(\varphi_v')^2(t_0)<0$. However, from the monotonicity of $u,v$, the definition of $L_0$, and \eqref{zxc1}, we have $\varphi_u'(t_0)>L_0\varphi_v'(t_0)\geq 0$. This gives the desired contradiction.
\end{proof}

\section{Proofs of Theorems \ref{ground}, \ref{vckn} and \ref{non}}\label{sec5}
The aim of this section is to investigate nonnegative ground states for the system \eqref{sys1}. We start by establishing the vector-valued Caffarelli-Kohn-Nirenberg inequality, whose Euler-Lagrange equation is precisely the system \eqref{sys1}.

\begin{proof}[Proof of Theorem \ref{vckn}]
    Utilizing the inequality \eqref{ckn}, we have
    \begin{equation}\label{qaz}
        \int_{\R^n}|x|^{-2a}(|\nabla u|^2+|\nabla v|^2)\geq S(a,b,n)\left(\left(\int_{\R^n}|x|^{-bp}|u|^p\right)^{\frac{2}{p}}+\left(\int_{\R^n}|x|^{-bp}|v|^p\right)^{\frac{2}{p}}\right).
    \end{equation}
    From the H\"older inequality, we have
    \begin{equation}\label{qaz1}
    \int_{\R^n}|x|^{-bp}|u|^\alpha |v|^\beta\leq \left(\int_{\R^n}|x|^{-bp}|u|^p\right)^{\frac{\alpha}{p}}\left(\int_{\R^n}|x|^{-bp}|v|^p\right)^{\frac{\beta}{p}}.
    \end{equation}
    Assuming
    \begin{equation}\label{qaz2}
        \int_{\R^n}|x|^{-bp}|u|^p=x_1^p,\quad \int_{\R^n}|x|^{-bp}|v|^p=x_2^p,
    \end{equation}
    we can obtain
    \begin{equation}\label{qaz3}
       \begin{aligned}
          \frac{\int_{\R^n}|x|^{-2a}(|\nabla u|^2+|\nabla v|^2)}{\left(\int_{\R^n}|x|^{-bp}(|u|^p+|v|^p+p\nu|u|^{\alpha}|v|^\beta)\right)^{\frac{2}{p}}}\geq S(a,b,n)f(x_1,x_2).
       \end{aligned}
    \end{equation}
    The equality holds if and only if both \eqref{qaz} and \eqref{qaz1} become equalities, indicating that $u$ and $v$ must be constant multiples of some minimizer of the inequality \eqref{ckn}. Since $f$ is a homogeneous function of degree zero, the minima always exists. Now it is easy to see that the sharp constant is given by $S(a,b,n)\min\limits_{x_1,x_2} f(x_1,x_2)$, and extremal manifold consists of  pairs $(c_1W,c_2W)$, where $W$ is a minimizer of the inequality \eqref{ckn} and $(c_1,c_2)$ is a minima of $f$.
\end{proof}
The characterization of nonnegative ground states follows straightforwardly from Theorem \ref{vckn}.
\begin{proof}[Proof of Theorem \ref{ground}]
    Suppose $(u,v)$ is a nontrivial solution to the system \eqref{sys1}. By multiplying the two equations by $u$ and $v$, respectively, and then adding the results, we deduce that
    \begin{equation}\label{qsc1}
        \int_{\R^n}|x|^{-2a}(|\nabla u|^2+|\nabla v|^2)=\int_{\R^n}|x|^{-bp}|u|^p+\int_{\R^n}|x|^{-bp}|v|^p+p\nu\int_{\R^n}|x|^{-bp}|u|^\alpha |v|^\beta.
    \end{equation}
    From Theorem \ref{vckn}, we have
    \begin{equation}\label{qsc2}
        \frac{\int_{\R^n}|x|^{-2a}(|\nabla u|^2+|\nabla v|^2)}{\left(\int_{\R^n}|x|^{-bp}(|u|^p+|v|^p+p\nu|u|^{\alpha}|v|^\beta)\right)^{\frac{2}{p}}}\geq \Bar{S}(a,b,n)
    \end{equation}
    Combining \eqref{qsc1} and \eqref{qsc2} indicates that
    \begin{equation}\label{qsc3}
         \int_{\R^n}|x|^{-2a}(|\nabla u|^2+|\nabla v|^2)\geq \Bar{S}(a,b,n)^{\frac{p}{p-2}}.
    \end{equation}
    The equality holds if and only $(u,v)$ takes the form $(sc_1W,sc_2W)$, where $W$ is a ground states for the equation \eqref{eq1}, $(c_1,c_2)$ is a minima of $f$ satisfying $c_1+c_2=1$, and $s$ is a normalization factor. Using \eqref{qsc1}, we obtain
    \begin{equation*}
        E(u,v)=\left(\frac{1}{2}-\frac{1}{p}\right)\int_{\R^n}|x|^{-2a}(|\nabla u|^2+|\nabla v|^2)\geq \left(\frac{1}{2}-\frac{1}{p}\right)\Bar{S}(a,b,n)^{\frac{p}{p-2}}.
    \end{equation*}
    Thus, $E(u,v)$ attains its minimum precisely when \eqref{qsc3} becomes an equality. Our assertions follow immediately.
\end{proof}
Let us give some clarification for Remark \ref{rem}. We focus on the three cases defined in \eqref{cas}.

In the first case: $\min\{\alpha,\beta\}<2$, without loss of generality, we assume $\alpha<2$. A crucial observation is the following basic inequality:
\begin{equation*}
    (1+x)^\epsilon>1+c(\epsilon)x,\quad\text{for any }0<x<1,\,0<\epsilon\leq1,
\end{equation*}
where $c(\epsilon)>0$ is a universal constant. Hence, for $0<x\ll 1$, we have
\begin{equation*}
    (1+x^p+p\nu x^\alpha)^\frac{2}{p}>1+c(p)p\nu x^\alpha>1+x^2,
\end{equation*}
implying that $f(x,1)<1$. Since $f(0,1)=f(1,0)=1$, we conclude that every minimum of $f$ is positive.

In the second case: $\min\{\alpha,\beta\}\geq2$ and $\nu>p^{-1}(2^{\frac{p}{2}}-2)$, one can directly check that $f(1,1)<1$, indicating the positivity of all nonnegative ground states.

In the third case: $\min\{\alpha,\beta\}\geq2$ and $\nu\leq \frac{p-2}{2p}$, utilizing the Bernoulli inequality
\begin{equation*}
    (1+x)^\epsilon<1+\epsilon x,\quad\text{for any }0<x,\,0<\epsilon<1,
\end{equation*}
we obtain
\begin{equation*}
    (1+x^p+p\nu x^\alpha)^{\frac{2}{p}}<1+\frac{2}{p}x^p+2\nu x^{\alpha}\leq1+x^2,\quad \text{when }0<x\leq 1,
\end{equation*}
which implies that $f(x,y)<1$ for any $0<x\leq y$. Similarly, one can show that $f(x,y)<1$ for any $0<y\leq x$. Thus, the minimum of $f$ can only be achieved at points $(x,0)$ and $(0,y)$.

In the following, we consider the nondegeneracy of positive synchronized solutions. Our proof relies on the following decoupled version given by Felli and Schneider in \cite{Fel}.
\begin{lemma}\label{lem}
    Assume $a<0,b_{\mathrm{FS}}(a)<b$ or $a\geq 0,b\neq 0$. Denote by $\{\lambda_k\}_{k=1}^\infty$ the eigenvalues of the problem
    \begin{equation*}
         -\div(|x|^{-2a}\nabla u)=\lambda|x|^{-bp}U^{p-2}u,\quad u\in D_a^{1,2}(\R^n),
    \end{equation*}
    where $U$ is defined in \eqref{bub}. Then, we have
    \begin{equation*}
        \lambda_1=1,\quad\lambda_2=p-1,\quad\lambda_3>p-1.
    \end{equation*}
    The corresponding eigenfunctions for $\lambda_1$ and $\lambda_2$ are given by $U$ and $\partial_{\mu}U$, respectively.
\end{lemma}
Our arguments below are inspired by \cite[Theorem 1.4]{Pen}.
\begin{proof}[Proof of Theorem \ref{non}]
    Without loss of generality, we assume $(u,v)=(c_1U,c_2U)$ with $U$ defined in \eqref{bub}. Suppose $(\varphi,\psi)$ is a nontrivial solution to the linearized system
    \begin{equation}\label{lin}
        \begin{cases}
            -\div(|x|^{-2a}\nabla \varphi)=|x|^{-bp}U^{p-2}(\theta_{11}\varphi+\theta_{12}\psi)\\
            -\div(|x|^{-2a}\nabla \psi)=|x|^{-bp}U^{p-2}(\theta_{21}\varphi+\theta_{22}\psi),
        \end{cases}
    \end{equation}
    where
    \begin{equation*}
        \theta_{11}=(p-1)c_1^{p-2}+\nu\alpha(\alpha-1)c_1^{\alpha-2}c_2^\beta,\quad \theta_{22}=(p-1)c_2^{p-2}+\nu\beta(\beta-1)c_1^\alpha c_2^{\beta-2}
    \end{equation*}
    and
    \begin{equation*}
        \theta_{12}=\theta_{21}=\nu\alpha\beta c_1^{\alpha-1}c_2^{\beta-1}.
    \end{equation*}
    Since $(c_1,c_2)$ is a solution to the system \eqref{re2}:
    \begin{equation*}
        \begin{cases}
            c_1^{p-2}+\nu\alpha c_1^{\alpha-2}c_2^\beta=1\\
            c_2^{p-2}+\nu\beta c_1^{\alpha}c_2^{\beta-2}=1,
        \end{cases}
    \end{equation*}
    we can simplify the representations of $\theta_{11}$ and $\theta_{22}$:
    \begin{equation*}
        \theta_{11}=p-1-\nu\alpha\beta c_1^{\alpha-2}c_2^\beta,\quad \theta_{22}=p-1-\nu\alpha\beta c_1^{\alpha}c_2^{\beta-2}.
    \end{equation*}
    Set $\gamma:=\frac{\theta_{11}-\theta_{22}-\sqrt{(\theta_{11}-\theta_{22})^2+4\theta_{12}^2}}{2\theta_{12}}=-\frac{c_2}{c_1}$. Multiplying the two equations in \eqref{lin} by 1 and $-\gamma$, respectively, and adding the results, we obtain
    \begin{equation*}
        -\div\left(|x|^{-2a}\nabla (\varphi-\gamma\psi)\right)=(p-1)|x|^{-bp}U^{p-2}(\varphi-\gamma\psi).
    \end{equation*}
    By Lemma \ref{lem}, we have $\varphi-\gamma\psi=\Lambda\partial_{\mu}U$ for some $\Lambda\in\R$. Thus, \eqref{lin} reduces to
    \begin{equation*}
    \begin{aligned}
        -\div(|x|^{-2a}\nabla \psi)=&\;|x|^{-bp}U^{p-2}\left(\theta_{21}\Lambda\partial_{\mu}U+(\theta_{22}+\theta_{21}\gamma)\psi\right)\\
        =&\;(p-1)\Lambda\nu\alpha\beta c_1^{\alpha-1}c_2^{\beta-1}\partial_{\mu}U\\
        &+(p-1-\nu\alpha\beta c_1^{\alpha-2}c_2^{\beta}-\nu\alpha\beta c_1^{\alpha}c_2^{\beta-2})|x|^{-bp}U^{p-2}\psi.
    \end{aligned}
    \end{equation*}
    Note that the nondegeneracy of $(u,v)$ is equivalent to the assertion that any solution to the system \eqref{lin} must be proportional to the pair $(c_1\partial_{\mu}U,c_2\partial_{\mu}U)$. It remains to verify whether we have
    \begin{equation*}
        p-1-\nu\alpha\beta c_1^{\alpha-2}c_2^{\beta}-\nu\alpha\beta c_1^{\alpha}c_2^{\beta-2}\neq \lambda_k\quad\text{for any }k\neq 2.
    \end{equation*}
    For $k\geq 3$, this holds clearly due to the fact that $\lambda_k>p-1$. For $k=1$, however, it is generally not evident to claim the incompatibility between the system \eqref{lin} and
    \begin{equation}\label{zzz}
        \nu\alpha\beta c_1^{\alpha-2}c_2^{\beta}+\nu\alpha\beta c_1^{\alpha}c_2^{\beta-2}= p-2.
    \end{equation}
    In the particular case $\nu\leq \frac{p-2}{2\alpha\beta}$, \eqref{zzz} does not hold due to the fact that $c_1,c_2<1$.
\end{proof}

\end{document}